\documentclass{article}
\usepackage{amsmath,amssymb,amsthm}
\usepackage[hidelinks]{hyperref}
\usepackage{yfonts}
\usepackage{enumerate}
\usepackage{color}
\usepackage{mathtools}
\usepackage{thmtools,thm-restate}
\usepackage{tikz}
\usepackage[justification=centering]{caption}
\usepackage{ wasysym } %% for \clock \diameter
\usepackage{verbatim}
\usepackage{bbm}

\usepackage{geometry}

\usepackage{graphicx} % Required for inserting images

\newtheorem{theorem}{Theorem}[section]
\newtheorem{lemma}[theorem]{Lemma}
\newtheorem{proposition}[theorem]{Proposition}
\newtheorem{corollary}[theorem]{Corollary}
\newtheorem{claim}[theorem]{Claim}
\newtheorem{fact}[theorem]{Fact}

\newtheorem{problem}[theorem]{Problem}

\title{Synchronizing random automata through repeated 'a' inputs}
\author{Anders Martinsson\thanks{Department of Computer Science, ETH Z\"urich, Switzerland\newline anders.martinsson@inf.ethz.ch}}
\date{\today}

\begin{document}

\maketitle

\begin{abstract}
In a recent article by Chapuy and Perarnau, it was shown that a uniformly chosen automaton on $n$ states with a $2$-letter alphabet has a synchronizing word of length $O(\sqrt{n}\log n)$ with high probability. In this note, we improve this result by showing that, for any $\varepsilon>0$, there exists a synchronizing word of length $O(\varepsilon^{-1}\sqrt{n \log n})$ with probability $1-\varepsilon$. Our proof is based on two properties of random automata. First, there are words $\omega$ of length $O(\sqrt{n \log n})$ such that the expected number of possible states for the automaton, after inputting $\omega$, is $O(\sqrt{n/\log n})$. Second, with high probability, each pair of states can be synchronized by a word of length $O(\log n)$.
\end{abstract}

\section{Introduction}
A deterministic finite automaton (from here on automaton) $\mathcal{A}=(Q, \Sigma, \delta)$ is a finite set $Q$ of states, a finite alphabet $\Sigma$, and a transition function $\delta:Q\times \Sigma \rightarrow Q$ that assigns a function $Q\rightarrow Q$ to each character in $\Sigma$. It is natural to extend the transition function to assign, for any word $\omega=\omega_1\omega_2\dots \omega_k\in \Sigma^*$, the transition function $\delta_\omega := \delta_{\omega_k}\circ \cdots\circ \delta_{\omega_1}$. Given a set $A\subseteq Q$ of states we further write $\delta_\omega(A)$ to denote $\{\delta_\omega(x) : x\in A\}$. A synchronizing word, also known as a reset sequence, is a word $\omega\in\Sigma^*$ such that $|\delta_\omega(Q)|=1$. An automaton is called synchronizable if it admits to such a word. 

A famous open problem about automata is to determine the smallest $C(n)$ such that any synchronizable automaton on $n$ states has a synchronizing word of length at most $C(n)$. In particular, the \v{C}ern\'y Conjecture proposes that $C(n)=(n-1)^2$. This question was first asked in an article by \v{C}ern\'y from 1964 \cite{Cerny}, in which it was shown that $(n-1)^2 \leq C(n) \leq 2^n-n-1$. Two years later, Starke \cite{Starke} improved the upper bound to $C(n)\leq O(n^3)$. In fact, the idea to prove $C(n)=O(n^3)$ is quite simple. By the pigeon hole principle, the shortest word that synchronizes any pair of states has length at most $O(n^2)$. Thus, repeating this $n-1$ times for suitably chosen pairs of states yield a synchronizing word of length $O(n^3)$. Over the last almost 60 years, constant factor improvements on this bound have been shown in \cite{Pin83, Fra82, Szy18} until the currently best known bound of $C(n)\leq 0.1654n^3+o(n^3)$ as given by Shitov \cite{Shi19} in 2019. 

As resolving the conjecture for general automata has shown to be a very challenging task, it is natural to consider the conjecture on special cases. One such case is to consider the behavior of a random automaton. We will here take $Q=[n]$, $\Sigma=\{a, b\}$, and let $\delta$ be chosen uniformly at random among all $n^{2n}$ possible transition functions. In other words, $
\delta$ independently maps each pair in $[n]\times \{a, b\}$ to a uniformly chosen state in $[n]$. In 2013, Cameron \cite{cam13} conjectured that a automaton obtained in this manner is synchronizable with high probability\footnote{We say that an event occurs with high probability if it occurs with probability $1-o(1)$ as $n\rightarrow\infty$.}. Indeed, this was proven affirmatively by Berlinkov \cite{Ber16}. Nicaud \cite{Nic19} refined this by showing that the automaton, with high probability, has a synchronizing word of length $O(n\log^3 n)$, thus in particular showing that the \v{C}ern\'y Conjecture holds for almost all automata. Very recently, this bound was further improved by Chapuy and Perarnau \cite{CP23} to $O(\sqrt{n}\log n)$ using an elaborate second moment argument.

Concerning lower bounds for the shortest synchronizing word in a random automata, it appears not much is rigorously known. It appears to be folklore in the area that, with high probability, the shortest synchronizing word has length at least $C n^{1/3}$, but to the authors knowledge this has never been published. Simulation studies \cite{KKS13, ST11, SZ22} indicate that the length increases as $\Theta(n^\alpha)$ for either $\alpha=\frac12$ or $\alpha=\frac12 + \varepsilon$ for some small $\varepsilon>0$, which seems to indicate that the bound by Chapuy and Perarnau is best possible, at least up to log factors. 

In this note, we will improve the result by Chapuy and Perarnau by showing that a random automaton typically has a synchronizing word of length $O(\varepsilon^{-1} \sqrt{n\log n})$ with probability $1-\varepsilon$. Alternatively, a random automaton has a synchronizing word of length $c_n \sqrt{n\log n}$ with high probability for any function $c_n$ that tends to infinity arbitrarily slowly in $n$. Compared to the result by Chapuy and Perarnau, this result has the benefit of having a considerably shorter proof.

We will build a synchronizing word in two phases. In the first phase, we aim to find a word $\omega$ such that $|\delta_\omega([n])|$ is not necessarily $1$, but at least much smaller than $n$. As a warm up example for the first phase, we consider the case where $\omega$ is a sequence of $\Theta(\sqrt{n\log n})$ 'a' transitions in a row.

\begin{theorem}\label{thm:aaa} Let $A:=\delta_{a^k}([n])$ denote the set of states that can be reached by performing $k:=\lceil 2 \sqrt{n \ln n}\rceil$ 'a'-transitions in a row. The expected size of $A$ is at most $(1+o(1))\sqrt{2\pi n}.$
\end{theorem}

Using the above result in phase one would allow us to match the bound by Chapuy and Perarnau. In order to make the additional factor $\sqrt{\log n}$ improvement, we modify $\omega$ by periodically inputting 'b':s every $\sqrt{n}$ positions.

\begin{theorem}\label{thm:aba} Let $A:=\delta_{\omega}([n])$ denote the set of states that can be reached by inputting the sequence $\omega:=a^{\lceil \sqrt{n}\rceil} (b(a^{\lceil \sqrt{n}\rceil}))^{\lceil \sqrt{\log_2 n}\rceil}$. The expected size of $A$ is $O(\sqrt{n/\log n}).$
\end{theorem}

Having reduced the number of possible states, in the second phase we proceed as in the proof by Starke by iteratively synchronizing pairs of states until only one possibility remains.

Determining the shortest word that synchronizes two states is morally similar to determining the shortest distance between two vertices in a random graph. Indeed, as one would expect from that paradigm, the shortest word is typically of order $\log n$.

\begin{theorem}\label{thm:pair} With high probability, for all pairs of states $x, y$ there exists a word $\omega=\omega^{xy}$ of length at most $3 \log_2 n$ such that $\delta_\omega(x)=\delta_\omega(y)$.
\end{theorem}

Putting all of this together, we obtain the following result.

\begin{corollary}\label{cor:synch} Let $X$ denote the length of the shortest synchronizing word. For any $0<\varepsilon\leq 1$ and any $n\geq n_0(\varepsilon)$ it holds that
$$ \Pr(X\leq C\varepsilon^{-1} \sqrt{ n\log n}) \geq 1-\varepsilon, $$
where $C>0$ is a universal constant independent of $\varepsilon$.
\end{corollary}
\begin{proof} Let $\omega$ be the string from Theorem \ref{thm:aba} and let $A:=\delta_\omega([n])$. Following the procedure above gives us, using Theorem \ref{thm:pair}, that $$\Pr(X \leq \lceil \sqrt{n} \rceil \cdot \lceil\sqrt{\log_2 n} + 1\rceil - 1 + |A| \cdot 3\log_2 n)=1-o(1).$$ Let $n_0(\varepsilon)$ be sufficiently large so that this expression is at least $1-\varepsilon/2$. Let $C$ be sufficiently large so that $\mathbb{E}[|A|]\leq \frac{C}{12} \sqrt{n/\log n}$ and $\lceil \sqrt{n} \rceil \cdot \lceil\sqrt{\log_2 n} + 1\rceil - 1 \leq \frac{C}2\sqrt{n \log_2 n}$ for all $n\geq n_0(\varepsilon)$. This is possible by Theorem \ref{thm:aba}.

Then, the event that $X>C\varepsilon^{-1}\sqrt{n\log_2 n}$ is contained in the union of the event that $X > \lceil \sqrt{n} \rceil \cdot \lceil\sqrt{\log_2 n} + 1\rceil - 1 + |A| \cdot 3\log_2 n$
and the event that $|A|>\frac{C}{6\varepsilon}\sqrt{n\log_2 n}$. The former has probability at most $\varepsilon/2$ by choice of $n_0(\varepsilon)$ and the latter has probability at most $\varepsilon/2$ by Markov's inequality. Hence, by the union bound, the inequality holds with probability at least $1-\varepsilon$, as desired.
\end{proof}

We end this section by some final comments. An immediate question given the above results is whether a more careful analysis could show that there exists a synchronizing word of length $O(\sqrt{n \log n})$ with high probability. For instance, this would follow if one could show that $|\delta_\omega([n])|$ from Theorem \ref{thm:aba} is concentrated around its expectation. However, we believe this to be false. On the other hand, it seems likely that a less structured string $\omega$, e.g. a random string, of length $\sqrt{n \log n}$ would reduce the number of possible states to $O(\sqrt{n/\log n})$ with high probability. However, we do not know how to approach this formally.

While upper bounds on the length of the shortest synchronizing word in a random automaton are now relatively well understood, finding matching lower bounds remains a challenging open problem. Proving a lower bound of $\Omega(n^{1/3})$ can be done roughly as follows. Fix a string $\omega$ of length $k$ and $k$ states $x_1, x_2, \dots, x_k$. For any state $x_i$, let $S_i$ denote the set of states reachable by starting in $x_j$ for some $j\neq i$ and following the transitions indicated by $\omega$. Then $S_i$ has size at most $(k-1)(k+1)<k^2.$ Thus, the probability that the trajectory starting at $x_i$ ever intersects $S_i$ is at most $k^3/n$. Picking $k$ sufficiently small so that $k^3/n<1/2,$ it follows that the probability that $\omega$ is synchronizing is at most $(k^3/n)^k=o(2^{-k}).$ Thus, by the union bound, with high probability no string of length $k$ is synchronizing. 

However, it appears that improving this by more than a constant factor requires new ideas. We pose this as an open problem.
\begin{problem} Show that the shortest synchronizing word in a random automaton is $\omega(n^{1/3})$ with high probability.
\end{problem}

\section{Random unary automata}

In this section, we prove Theorem \ref{thm:aaa}. In order to do this, let us consider the structure of the random unary automaton on $n$ states $([n],\{a\}, \delta_a)$ formed by restricting the above automaton to the alphabet $\{a\}$. We think of this as a directed graph $D$ on vertex set $[n]$ where for each $x\in [n]$ we add an arc to $\delta_a(x)$.

It is not too hard to convince oneself that any weakly connected component in $D$ contains a unique directed cycle (which may possibly be a self-loop) with the remaining vertices forming trees directed towards the cycle. We will refer to a vertex as \emph{cyclic} if it is contained in such a cycle, and \emph{non-cyclic} otherwise.

A vertex $x\in D$ is reachable after applying $k$ 'a':s in a row if it is the end-point of a walk in $D$ of length $k$. This can happen in two ways. Either $x$ is cyclic, or it is the end-point of some directed path in $D$ of length $k$. We will prove Theorem \ref{thm:aaa} by estimating the expected number instances of each of these.

Let us start by bounding the expected number of cyclic states. For a given state $x_0$ we can check whether or not it is cyclic by iteratively following the unique out-going arc from the current vertex to form the sequence  $x_0, x_1, 
\dots$ until we close a cycle, that is, the first step $t$ where $x_t=x_{t'}$ for some $t'<t$. The state $x_0$ is cyclic if and only if the the edge that closes the cycle goes back to $x_0$, that is $t'=0$. 

For any $t\geq 0$, let $P_t$ denote the probability that we have not yet closed a cycle after $t$ steps. We have
$$ P_t = \left( 1 - \frac{1}{n}\right)\left( 1 - \frac{2}{n}\right)\cdots\left( 1 - \frac{t-1}{n}\right) \leq \exp\left(-\frac{{t \choose 2}}{n} \right) \leq \exp\left(-\frac{(t-1)^2}{2n} \right). $$
By the law of total probability, we get
$$\Pr(x_0\text{ is cyclic}) = \sum_{t=0}^{n-1} P_t \cdot \frac{1}{n} \leq \frac{1}{n} \sum_{t=0}^\infty \exp\left(-\frac{(t-1)^2}{2n} \right), $$
where $$\sum_{t=0}^\infty \exp\left(-\frac{(t-1)^2}{2n} \right)\sim \int_0^\infty \exp\left(-\frac{t^2}{2n} \right)\, dt = \sqrt{2n}\int_0^\infty e^{-t^2}\,dt=\sqrt{2\pi n}.$$
This yields $\Pr(x_0\text{ is cyclic})\leq (1+o(1))\sqrt{2\pi/n}.$ Thus, the expected number of cyclic states is at most $(1+o(1))\sqrt{2\pi n}$.

Second, we can upper bound the expected number of vertices that are end-points of some path of length $k$ by the expected number of paths of length $k$, which is $n\cdot P_k \leq n \cdot \exp\left(-(k-1)^2/2n\right)$. This is clearly $o(1)$ if $k\geq 2\sqrt{n\log n}$. \qed

\section{Improving the first phase}

In this section, we prove Theorem \ref{thm:aba}. In order to do this, we consider the following random structure. Given a vertex set $V$ and a probability vector $(p_v)_{v\in V}$ of non-negative real numbers whose sum is $1$, we let $G(V, (p_v)_{v\in V})$ denote a random $1$-out-regular digraph on vertex set $V$ where the end-points of the $n$ edges are chosen independently with distribution according to $(p_v)_{v\in V}$.

\begin{theorem}\label{thm:inhom}
Given any vertex set $V$ and probability vector $(p_v)_{v\in V}$, the expected number of cyclic vertices in $G(V, (p_v)_{v\in V})$ is at most $(1+o(1))\sqrt{2\pi |V|}$, and the expected number of vertices that are end-points of paths of length $k$ is $O(|V|/k)$.
\end{theorem}

Before proving this theorem, let us see how it can be used to prove Theorem \ref{thm:aba}.

\begin{proof}[Proof of Theorem \ref{thm:aba}.]
Let $A':=\delta_{a^{\lceil \sqrt{n}\rceil}}([n])$ denote the set of states in a random automaton $\mathcal{A}=([n], \{a, b\}, \delta)$ that are reachable after a sequence of $\lceil \sqrt{n} \rceil$ 'a' inputs. Then $\mathbb{E}[|A'|]=O(\sqrt{n})$ by Theorem \ref{thm:inhom}. Moreover, observe that $\delta_{b(a^{\lceil \sqrt{n}\rceil})}(A')\subseteq A'.$ Hence, we can define a unary automaton $\mathcal{A}'$ on states $A'$ and with transition function $\delta'$ where $\delta'(x):=\delta_{b(a^{\lceil \sqrt{n}\rceil})}(x)$.

Conditioned on $\delta_a$ we have that the values of $\delta'(x)$ for $x\in A'$ are independent and identically distributed, as they are completely determined by the choice of $\delta_b(x)$. Hence $\mathcal{A}'$ satisfies the condition of Theorem \ref{thm:inhom}, and, as a consequence, the expected size of $\delta'^{(\lceil \sqrt{\log n}\rceil)}(A')$ is $O(\sqrt{|A'|}+|A'|/\sqrt{\log n})$. Averaging over the choice of $\delta_a$, simplifies this to $O(\sqrt{n/\log n})$. The statement follows by observing that $A=\delta_\omega([n])=\delta'^{(\lceil \sqrt{\log n}\rceil)} \circ \delta_{a^{\lceil \sqrt{n}\rceil}}([n])=\delta'^{(\lceil \sqrt{\log n}\rceil)}(A')$.
\end{proof}

In the remainder of the section, we will prove Theorem \ref{thm:inhom}. Let us start by considering the expected number of cyclic states/vertices. As we have already shown in the last section that the expected number of cyclic vertices have the desired bound for the probability vector $(1/n, \dots, 1/n),$ the following statement suffices to extend this bound to all probability vectors.

\begin{claim} For a given $n$, the uniform probability vector $(1/n, \dots, 1/n)$ is the vector that maximizes the expected number of cyclic vertices.
\end{claim}
\begin{proof}
The expected number of cyclic vertices can be written as $$\sum_{x\in V} \sum_{x\in C \subseteq V} (|C|-1)! \prod_{y\in C} p_y= \sum_{\emptyset\neq C\subseteq V} |C|! \prod_{y\in C} p_y.$$
This is a polynomial in $(p_v)_{v\in V}$ with non-negative coefficients. For any two distinct states $x, y$ we can group the monomials as $a \cdot p_x \cdot  p_y + b \cdot (p_x + p_y) + c$ where $a, b, c$ are polynomials with non-negative coefficients with variables $(p_z)_{z\neq x, y}.$ Moreover, as $a$ contains the monomial $2$ coming from the term $C=\{x, y\},$ it is strictly positive. This means that, for a given value of $p_x+p_y$, the expression is maximized only when $p_x=p_y$. Hence, subject to $\sum_x p_x = 1$, $(1/n, \dots, 1/n)$ is the unique maximizer.
\end{proof}

In order to bound the expected number of vertices that are end-points of long paths, the idea is to relate $G(V, (p_v)_{v\in V})$ to a Galton-Watson process where the number of children of a given node follows a Poisson distribution with parameter $1$.

Define the sequence $(q_k)_{k=0}^\infty$ according to $q_0=0$ and $q_{k+1}=\exp(-(1-q_k))$ for any $k\geq 0$. This denotes the probability that the Galton-Watson process reaches extinction before the $k$th generation.
\begin{fact} $q_k = 1-O(1/k)$.\end{fact}

We aim to show that the expected number of vertices in $G(V, (p_v)_{v\in V})$ that are end-points of paths of length $k$ is at most $n(1-q_k).$ It turns out to be easier to prove a strengthened version, from which Theorem \ref{thm:inhom} directly follows. For a digraph $G$, a vertex $v$ and a vertex set $S$, let $d(v\rightarrow S)$ denote the length of the shortest path from $v$ to $S$ in $G$ (or $\infty$ if no such path exists).

\begin{proposition}\label{prop:ind} Let $S\subseteq V$ denote a random set of size $\ell$ chosen uniformly at random. With probability at least $q_k^\ell$, $G(V, (p_v)_{v\in V})$ contains no vertex $v\in V$ such that $d(v\rightarrow S) = k.$
\end{proposition}

To show this, we use the following inequality.
\begin{lemma}\label{lem:standardineq} For any $0< x \leq 1$ and integers $0 \leq a \leq b$, we have $$ \left(1-\frac{a}b x\right)^{b-a} \geq e^{-ax}. $$
\end{lemma}
\begin{proof} By log-concavity, it suffices to verify the inequality for $x=0$ and $x=1$. The inequality clearly holds for $x=0$. Moreover, as $\left(1 - \frac{a}b\right)^{-1} = 1+\frac{a}{b-a} \leq e^{a/(b-a)}$, where the last step follows by the inequality $e^x\geq 1+x$, we find that that $\left(1-\frac{a}b x\right)^{b-a} \geq e^{-a}$.
\end{proof}

\begin{proof}[Proof of Proposition \ref{prop:ind}.] We show this by induction on $k$. The statement is trivially true for $k=0$. Let $k\geq 1$, and assume the statement holds for all $k'<k$.

Let us consider the probability of the prescribed event for a given set $S$. Each vertex of $V\setminus S$ generates its out-going edge independently so that it connects to $S$ with probability $p(S)$. Let $S'$ denote the set of these vertices. Conditioned on the event that a vertex $v\in [n]\setminus S$ is not in $S'$, it connects to a vertex in $[n]\setminus S$ according to the probability distribution $(p_i/(1-p(S)))_{i\in V\setminus S}.$ 

Thus, by coupling this to the $1$-out-regular graph $G'=G(V\setminus S, ( p_v/(1-p(S)))_{v\in V\setminus S}),$ we see that the probability that there exists a vertex in $G(p_1, \dots, p_n)$ with distance exactly $k$ to $S$ is the same as the probability that there exists a vertex in $G'$ with distance exactly $k-1$ to $S'$. Hence, by the induction hypothesis, we get
$$\Pr\left( \neg \exists v : d(v\rightarrow S) = k | S \right) \geq \sum_{\ell'=0}^{n-\ell} 
{n-\ell \choose \ell'} p(S)^{\ell'}(1-p(S))^{n-\ell-\ell'} q_{k-1}^{\ell'} = \left( 1-p(S)(1-q_{k-1})\right)^{n-\ell}.$$
Observe that the right-hand side is convex in $p(S)$. As $\mathbb{E}[p(S)] = \ell/n$ for a uniformly chosen $\ell$-set $S$, it follows by Jensen's inequality that
$$\Pr\left( \neg \exists v : d(v\rightarrow S) = k \right) \geq \left( 1-\frac{\ell}{n}(1-q_{k-1})\right)^{n-\ell} \geq \exp(-\ell(1-q_{k-1})),$$
where, in the last step, we use Lemma \ref{lem:standardineq}. As $e^{-\ell(1-q_{k-1})} = q_k^\ell$, this implies the induction statement, which concludes the proof.
\end{proof}

\section{Synchronizing two states}

In this section, we prove Theorem \ref{thm:pair}. Let us start by introducing some additional notation. First, for any two states $x, y$, we let $d(x, y)$ denote the length of the shortest word $\omega$ such that $\delta_\omega(x)=\delta_\omega(y)$. If no such word exists we put $d(x, y)=\infty$. To aid the analysis below, we will independently color each pair in $[n]\times\{a, b\}$ either red with probability $p$ or blue with probability $1-p$ where $p:=1/\log^2 n$. For any two states $x, y$, we let $d_b(x, y)$ denote the length of the shortest word $\omega$ such that $\delta_\omega(x)=\delta_\omega(y)$ and where the two trajectories only use blue transitions. Again, we put $d_b(x, y)=\infty$ if no such word exists. Let $$L_b:=\{(x, y)\in [n]^2 : d_b(x, y) > 2 \log n\}.$$
Note in particular that the outcome of $L_b$ is independent of $\delta_c(x)$ for any red pair $(x, c)$.

The proof of Theorem \ref{thm:pair} uses a double exposure argument. We first uncover $\delta_c(x)$ for all blue pairs $(x, c)$. This information lets us determine $L_b$. A second moment argument shows that, with high probability, most pairs of states are not in $L_b$. Second, for any pair of states $(x, y)$ we show using an exploration process that, with probability $1-O(1/n^2)$ there are many words $\omega$ of length $O(\log n)$ such that both, say, $(\delta_\omega(x), a)$ and $(\delta_\omega(y), a)$ are red. By uncovering the image of all red transitions, it follows by very high probability that there is some $\omega$ for which $(\delta_{\omega a}(x), \delta_{\omega a}(y))$ is not in $L_b$, thus we can synchronize $x$ and $y$ by adding an additional $2\log n$ steps. Theorem \ref{thm:pair} follows relatively easily using Markov's inequality.

\begin{proposition}\label{prop:weakpair} For any fixed pair of states $x, y$, $\Pr(d(x, y) \leq  2 \log_2 n )= 1-o(1). $
\end{proposition}
\begin{proof} We may assume that $x\neq y$, as otherwise the statement is obviously true. For any $\omega \in \{a, b\}^k$, let $X_\omega$ denote the indicator function for the event that the sequences of states formed by following $\omega$ from $x$ and $y$ respectively have the same end-point, but where the states involved are otherwise pair-wise distinct. Let $X=\sum_\omega X_\omega$. We will show using the second moment method that $\Pr(X>0)=1-o(1).$

By direct counting, we have $\mathbb{E}X_\omega = (n-2)(n-3)\cdots (n-2k)/n^{2k} = (1-O(\frac{k^2}{n}) )\frac{1}{n}$. Moreover, for any distinct $\omega, \omega'\in \{a, b\}^k$, we claim that $\Pr(X_{\omega'}=1|X_\omega=1) \leq \frac{1}{n} + O(\frac{k^4}{n^2}).$ To see this, condition on $X_\omega=1$ and the set of states $A$ that appears along the trajectories from $x$ to $\delta_\omega(x)$ and from $y$ to $\delta_{\omega}(y)$. Consider the sequences of states
$$(x, y)=(x_0, y_0), (x_1, y_1), \dots, (x_{k-1}, y_{k-1})$$
starting at $x$ and $y$ and following the transitions as indicated by $\omega'$ up until the second to last character. Letting $0\leq \ell < k$ denote the length of the longest common prefix of $\omega$ and $\omega'$, we know that $(x_0, y_0), \dots (x_\ell, y_\ell)$ deterministically follow the corresponding sequences in $A$ formed by $\omega$. Let us consider three cases for how the probability that $X_{\omega'}=1$ depends on the behavior of the $\omega'$-sequences.
\begin{enumerate}
    \item If either of the sequences $x_0, x_1, \dots, x_{k-1}$ or $y_0, y_1, \dots, y_{k-1}$ contains a state more than once, or the sequences have a common state, then $X_{\omega'}=1$ with probability $0$.
    \item Otherwise, if both sequences $x_{\ell+1}, \dots, x_{k-1}$ and $y_{\ell+1}, \dots, y_{k-1}$ contain states in $A$, then (trivially) $X_{\omega'}=1$ with probability at most $1$.
    \item Otherwise, the probability that $X_{\omega'}=1$ is at most $1/n$ (as at least one of $x_k$ and $y_k$ is still chosen uniformly at random, given everything else observed). 
\end{enumerate}
Thus, we can bound the conditional probability that $X_{\omega'}=1$ given $X_\omega=1$ by $1/n$ plus the probability that Case $2.$ occurs, which can be readily bounded by $O((k^2/n)^2)$.

It follows that $\mu:=\mathbb{E}X= \sum_\omega \mathbb{E}X_\omega = 2^k/n - O(2^k k^2/n^2)$ and $\mathbb{E}X^2=\sum_{\omega}\sum_{\omega'}\mathbb{E}X_\omega X_{\omega'} = \mu + \sum_{\omega}\sum_{\omega'\neq \omega}\Pr(X_{\omega'}=1|X_\omega=1)\Pr(X_\omega=1)\leq \mu+2^{2k}/n^2 + O(2^{2k} k^4/n^3),$ which gives Var$(X)=O(\mu + \mu^2 k^4/n)$. Now, letting $k=\lfloor 2\log n \rfloor$ so that $\mu=\Theta(n)$, it follows by Chebyshev's inequality that $\Pr(X>0) \geq 1-O(\frac{\log^4 n}{n})$, as desired.
\end{proof}

Observe that as $p=o(1/\log n)$, we should expect most paths of length $O(\log n)$ to consist only of blue edges. Combining this with the previous proposition immediately implies information about the size of $L_b$.

\begin{corollary}\label{cor:Lb} $\Pr(|L_b| \geq n^2/2 )=o(1)$
\end{corollary}
\begin{proof}
Given that $d(x, y)\leq k$ for some states $x, y$ it follows that $d_b(x, y)\leq k$ provided that all $\leq 2k$ edges involved are blue. Thus $$\Pr(d_b(x, y)\leq k\,|\,d(x, y)\leq k) \geq (1-p)^{2k}.$$ Applying Proposition \ref{prop:weakpair} it follows that $$\Pr(d_b(x, y)\leq 2\log n) \geq \Pr(d(x, y)\leq 2\log n)(1-p)^{4\log n} = 1-o(1).$$ By Markovs inequality we have $$\Pr(|L_b| \geq n^2/2 )\leq \mathbb{E}|L_b|/(n^2/2)=\sum_{x,y} \Pr(d_b(x, y)> 2\log n)/(n^2/2) = o(n^2/n^2)=o(1).$$ 
\end{proof}

\begin{proposition}\label{prop:explore} For any fixed pair of states $x, y$, with probability $1-O(1/n^2)$, one of the following occurs
\begin{enumerate}[(i)]
    \item $d(x, y)\leq 2.5 \log n$,
    \item $|L_b|\geq n^2/2$.
\end{enumerate}
\end{proposition}
\begin{proof} Given distinct states $x, y$, we consider the following exploration process. We initially let $S_0=\{""\}$ be the set containing the empty string and mark $x$ and $y$ as visited. For each $t=1, \dots \lfloor \frac12 \log n\rfloor-1$, we generate the set $S_t$ by, for each $\omega \in S_{t-1}\times \{a, b\}$, checking whether $\delta_\omega(x)$ and $\delta_\omega(y)$ have previously been visited in the exploration. If both were previously unvisited, we add $\omega$ to $S_t$ (otherwise, we discard $\omega$). In either case, we mark both $\delta_\omega(x)$ and $\delta_\omega(y)$ as visited, and continue.

We claim that, with probability $1-O(1/n^2)$, the process will end up with the set $S_{\lfloor \frac12 \log n\rfloor-1}$ having size $\Omega(\sqrt{n})$. Given any history of the process, the probability that a string $\omega$ is discarded is at most $2/n$ times the number of previously visited vertices. This means that, with probability at least $1-O(1/n^2)$, we will discard at most one string while generating, say, $S_1$ and $S_2$. Similarly, using the more liberal bound of $O(\sqrt{n})$ vertices visited throughout the process, the probability of discarding a string in any later step is $O(1/\sqrt{n})$, so, with probability $1-O(1/n^2)$, at most three strings will be discarded throughout the whole process. It is not too hard to convince oneself that, under those restrictions $|S_t|\geq 2^{t-2}$ for all $t$. In particular $|S_{\lfloor \frac12 \log n\rfloor-1}|=\Omega(\sqrt{n}).$

Let us now check for all $\omega\in S_{\lfloor \frac12 \log n\rfloor-1}$ which color the pairs $(\delta_\omega(x), a)$ and $(\delta_\omega(y), a)$ have. Assuming the exploration process succeeded, it follows by Chernoff bounds that, with probability $1-\exp(-\Omega(\sqrt{n}p^2 ))$, there are $\Omega(\sqrt{n}p^2)$ strings $\omega\in S_{\lfloor \frac12 \log n\rfloor}$ such that both of the corresponding pairs are red.

Assuming all of this has succeeded, which occurs with probability $1-O(1/n^2)$, we now uncover the color of every pair $(x, c)$ and the value of any blue pair $\delta_c(x)$ that has not been explored previously in the process. Note that this determines $L_b$.

Now, it either turns out that $|L_b|\geq n^2/2$, or each of the aforementioned red pairs have a probability of at least $1/2$ or mapping to a pair of states $(x', y')$ such that $d_b(x', y')\leq \lfloor 2 \log n \rfloor$. The probability that this happens for at least one such pair is $1-2^{-\Omega(\sqrt{n}p^2)}=1-o(1/n^2)$, in which case $d(x, y) \leq 2.5 \log n,$ as desired.
\end{proof}

We remark that the only way the exploration process in the proposition above fails with a probability greater than $O(1/n^3)$ is if both strings at $t=1$ are discarded. A more elaborate exploration process can be made to work with probability $O(1/n^3)$ if one allows as a third possible outcome that either $x$ or $y$ is isolated.

\begin{proof}[Proof of Theorem \ref{thm:pair}.] Let $k$ be any sufficiently slowly increasing function in $n$, say, $k=\frac12 \log n$. It follows from Corollary \ref{cor:Lb} and Proposition \ref{prop:explore} that, with probability $1-o(1)-O(1/k)=1-o(1)$, all but at most $k$ pairs of states $(x, y)$ satisfy $d(x, y)\leq 2.5 \log n$.

Suppose this holds and let $(x, y)$ be a pair of states such that $d(x, y)>2.5\log n$. Consider the number of states that can be reached from $x$ in $k$ steps. If there are more than $k$ such states (including $x$ itself), then one of them has to correspond to a pair of states satisfying $d(x', y')\leq 2.5\log n$, in which case $d(x, y)\leq 2.5\log n + k$. If there are at most $k$ reachable states (including $x$ itself), then no further states can be reached from $x$ even in more steps. We can easily bound the expected number of such terminal sets of states by
$\sum_{\ell=1}^k n^\ell \cdot \left(\frac{\ell}{n}\right)^{2\ell} = O(1/n).$ Thus, with probability $1-O(1/n)$ no such sets exist. We conclude that, with high probability, $d(x, y)\leq 2.5 \log n + k = 3\log n$ for all $x, y$, as desired.
\end{proof}

\bibliographystyle{abbrv}
\bibliography{references}
\end{document}